\documentclass[11pt]{amsart}

\usepackage[T1]{fontenc}
\usepackage[utf8]{inputenc}
\usepackage[british]{babel}
\usepackage{amsmath,amssymb,amsfonts,amsthm}
\usepackage{mathtools}
\usepackage{enumitem}
\usepackage{url}
\usepackage{microtype}

\newcommand{\N}{\mathbb N}
\newcommand{\E}{\mathbb E}

\newcommand{\calG}{\mathcal G}
\newcommand{\calP}{\mathcal P}
\newcommand{\calQ}{\mathcal Q}
\newcommand{\calS}{\mathcal S}

\newcommand{\TV}{\operatorname{TV}}

\newtheorem{theorem}{Theorem}[section]
\newtheorem*{theorem*}{Theorem}

\newtheorem{lemma}[theorem]{Lemma}

\newtheorem{definition}[theorem]{Definition}

\theoremstyle{remark}

\title[Star observations in bounded-degree graphs]{Star observations in bounded-degree graphs}
\author{Bal\'azs Szegedy}
\date{\today}

\address{Rényi Institute of Mathematics, Budapest, Hungary}
\email{szegedyb@gmail.com}


\begin{document}

\begin{abstract}
Similarity metrics are central in the theory of large networks and graph limits. For bounded-degree graphs, the Benjamini--Schramm metric records the distribution of rooted neighbourhoods, while the stronger colored-neighbourhood metric \cite{BR} gives rise to local-global convergence \cite{HLSz}. In this paper we show that this intricate topology is already determined by much smaller observations. For technical convenience and greater generality, we work with graphings, which are measurable generalizations of finite graphs and include all finite graphs as special cases. We prove that, for graphings of uniformly bounded degree, convergence of all colored degree distributions, or equivalently of all colored star statistics, is equivalent to local-global convergence. We also introduce an even more economical sampling procedure, the colored cherry metric, in which one observes only the root and two randomly chosen neighbours, and prove that it induces the same topology. Thus the full local-global structure can be reconstructed, at the level of topology, from families of very small colored observations. Our star-observation theorem was previously announced in the work of Backhausz and the author \cite{BSzAction} as an important ingredient in the proof that the so-called action convergence unifies dense graph limit theory with local-global convergence, thereby providing a general graph limit theory for sparse, dense, and intermediate-density graphs.
\end{abstract}

\maketitle

\section{Introduction}

A fundamental problem in graph limit theory \cite{L} is to decide when two large finite graphs should be considered similar.  Different answers lead to different limit theories.  For dense graphs, one may sample a fixed number of uniformly random vertices and record the induced subgraph \cite{LovaszSzegedyDense}. Similarity is then measured by comparing the resulting finite subgraph distributions.  This approach is powerful in the dense setting, but it degenerates for sparse graphs: if a graph has $o(n^2)$ edges, then a uniformly sampled set of $k$ vertices typically spans no edges.

For bounded-degree graphs, the natural replacement is local sampling.  Instead of sampling an induced subgraph on random vertices, one chooses a uniformly random root and observes its radius-$r$ induced neighbourhood.  The Benjamini--Schramm \cite{BS} topology is generated by the distributions of these rooted neighbourhoods as $r$ varies.

A substantial refinement appears when one allows arbitrary vertex colorings before making the observation. In the framework of so-called local-global convergence \cite{BR,HLSz,NesetrilOssonaLocalGlobal}, one compares, for every number of colors $k$ and every radius $r$, the set of all possible distributions of rooted $k$-colored radius-$r$ neighbourhoods. This topology is substantially stronger than ordinary Benjamini--Schramm convergence: it records not only the local structure of the graph, but also how this structure interacts with arbitrary global partitions of the vertex set. Beyond local properties, it detects important graph invariants such as the independence ratio, cut densities, and large-scale expansion.

Graphings \cite{Levitt,L,HLSz} provide a natural extension of graph theory in the limit theory of bounded-degree graphs. They are bounded-degree Borel graphs on probability spaces satisfying a measure-preserving condition, and they arise as limit objects for Benjamini--Schramm convergent graph sequences \cite{ElekGraphings} as well as for local-global convergent sequences \cite{HLSz}. In this paper we develop our results in the general setting of graphings. This not only gives a more general statement, but also leads to cleaner and more transparent proofs.

The main result of this paper (see Theorem \ref{thm:main}) is that, for bounded-degree graphings, the full local-global topology is already determined by non-induced radius-one observations.  Equivalently, it suffices to observe all possible colored stars.  A colored star records the color of a vertex and the numbers of its neighbours of each color; it does not record edges among neighbours.  Thus throughout the paper we distinguish between full induced rooted radius-$r$ balls, used in the definition of local-global convergence, and non-induced rooted colored stars, used in the main theorem. It is surprising that such star data determine the entire local-global topology. It is not even a priori clear that they determine the Benjamini--Schramm topology. 

We also consider a still smaller-looking observation, which we call the colored cherry statistic. Given a colored graph, one chooses a random root. If the root has degree zero, one observes only the rooted colored vertex; if it has degree one, one observes the rooted colored edge; and if it has degree at least two, one observes the root together with two uniformly chosen distinct neighbours. A single cherry sees much less than a star, since it only samples a pair of neighbours rather than the whole multiset of neighbour colors. The second main result of the paper (see Theorem \ref{thm:cherry-equivalence}) is that, after allowing arbitrary colorings, colored cherries determine the same topology as colored stars. Hence colored cherries, colored stars, and colored neighbourhoods all give topologically equivalent notions of convergence for bounded-degree graphings. 

Theorem~\ref{thm:cherry-equivalence} on cherry observations is, in a natural sense, optimal. If cherries are replaced by the still smaller observation consisting of a single rooted coloured edge, then the resulting statistics recover only certain quotient data of the graph. Graph-limit notions based on graph quotients can be interesting even beyond the bounded-degree setting \cite{KKLS}. However, examples show that such quotient data are not sufficient to recover the local structure of bounded-degree graphs. A striking feature of the coloured cherry convergence introduced here is that its sample space does not depend on the maximum degree. It therefore extends naturally, and compactly, to arbitrary graph sequences. This opens up a possible new direction in the theory of general graph limits, based on observations of paths of length two and related generalized quotients.

Our results are also useful beyond the internal comparison of local graph metrics. In the paper \emph{Action convergence of operators and graphs} \cite{BSzAction}, Backhausz and the author of the present paper introduced action convergence, a functional-analytic framework for graph and operator limits. One of the main claims of that work is that action convergence unifies dense graph limit theory with sparse local-global convergence. The star-observation theorem proved here (and announced in \cite{BSzAction}), is an important ingredient in the proof of the local-global part of that unification. 

\begin{definition}[colored degree distribution]
Let $G$ be a finite graph and let $c\colon V(G)\to [k]$ be a $k$-coloring, where $[k]=\{1,\ldots,k\}$.  The \emph{colored degree} of a vertex $v\in V(G)$ is
\[
 d_c(v)=\bigl(c(v),d_1(v),\ldots,d_k(v)\bigr)\in [k]\times\N^k,
\]
where $d_i(v)$ is the number of neighbours of $v$ whose color is $i$.  The \emph{colored degree distribution} of $(G,c)$ is the distribution of $d_c(v)$ when $v$ is chosen uniformly at random from $V(G)$.
\end{definition}

When $k=1$, this is just the ordinary degree distribution.  For larger $k$, however, the family of all colored degree distributions is much richer.  It is precisely this family that defines the colored star statistics of the graph.

The two main theorems proved below may be stated informally as follows.

\begin{theorem*}
For graphings of uniformly bounded degree, convergence of all colored star statistics is equivalent to local-global convergence. The same topology is also induced by convergence of all colored cherry statistics.
\end{theorem*}

The proof is organised as follows. Section~\ref{sec:local-global} recalls local-global convergence and graphings and defines the star statistics used in the main theorem. Section~\ref{sec:proof} proves the star theorem. Section~\ref{sec:cherry} proves the cherry theorem. 

\section{Local-global convergence, graphings and stars}\label{sec:local-global}

All graphs and graphings in this paper are simple, undirected, and loopless. Fix an integer $d\geq 2$. Let $\calG_d$ denote the class of finite graphs of maximum degree at most $d$, considered up to isomorphism. A rooted graph is a graph together with a distinguished vertex, called the root and denoted by $o$.

Recall that a $k$-coloring of a graph $G=(V,E)$ is a function $c\colon V\to [k]$, where $[k]=\{1,2,\dots,k\}$. Occasionally, it will be convenient to identify the color set with another finite set; this will always be clear from the context.

For integers $k\geq 1$ and $r\geq 0$, let $\calG_{d,k,r}$ be the finite set of isomorphism classes of rooted $k$-colored graphs of maximum degree at most $d$ in which every vertex lies at distance at most $r$ from the root. In particular, $\calG_{d,k,0}$ is the set of colored one-vertex rooted graphs.

Assume that $\Omega$ is a finite set. We write $\calP(\Omega)$ for the simplex of probability measures on $\Omega$. Recall that the total variation distance on $\calP(\Omega)$ is defined by
\[
 d_{\TV}(\mu,\nu)=\frac12\sum_{\alpha\in\Omega}|\mu(\alpha)-\nu(\alpha)|.
\]
If $A,B\subseteq\calP(\calG_{d,k,r})$, their Hausdorff distance (induced by the metric $d_{\TV}$) is denoted by $d_H(A,B)$.

Let $G=(V,E)\in\calG_d$ and let $f\colon V\to[k]$ be a coloring.  For $v\in V$, let
\[
 \gamma_{r,G,f}(v)\in\calG_{d,k,r}
\]
be the isomorphism class of the rooted, induced, colored radius-$r$ neighbourhood of $v$.  We write $\tau_r(G,f)$ for the distribution of $\gamma_{r,G,f}(v)$ when $v$ is uniformly random in $V$.  Finally, define the set
\[
 Q_{k,r}(G)=\{\tau_r(G,f)\,\bigm|\,f\colon V(G)\to[k]\} \subseteq \calP(\calG_{d,k,r}).
\]

\begin{definition}[Local-global convergence for graphs]
A sequence $(G_n)_{n\geq 1}$ in $\calG_d$ is \emph{local-global convergent} if, for every $k\geq 1,r\geq 0$, the sets $Q_{k,r}(G_n)$ converge in Hausdorff distance.
\end{definition}

Note that for finite graphs we have that $Q_{k,r}(G)$ is a finite set and hence it is closed. It will be important in this paper that local-global convergence extends naturally to graphings.

\begin{definition}[Graphing]
Let $X$ be a Polish topological space and let $\nu$ be a probability measure on the Borel
sets in $X$. A \emph{graphing} of maximum degree at most $d$ is a Borel graph $G=(X,E)$ such that every vertex has degree at most $d$ and
\begin{equation}\label{eq:mass-transport}
 \int_A e(x,B)\,d\nu(x)=\int_B e(x,A)\,d\nu(x)
\end{equation}
for all measurable sets $A,B\subseteq X$, where $e(x,S)$ denotes the number of neighbours of $x$ that belong to $S$.
\end{definition}

Condition~\eqref{eq:mass-transport} is the mass-transport principle.  It says that the expected number of edges sent from $A$ to $B$ equals the expected number sent from $B$ to $A$.  Finite graphs with the uniform measure on their vertex sets are examples of graphings.

We shall use the standard measurable structure of bounded-degree Borel graphs.  In particular, by Lusin--Novikov uniformization \cite{KechrisCDST} one may measurably enumerate neighbours locally; consequently finite-radius balls of measurable sets are measurable, and for every measurable finite coloring $f$ the local type map $x\mapsto\gamma_{r,G,f}(x)$ is measurable.  These standard facts will be used without further comment.

For a graphing $G=(X,\nu,E)$, a $k$-coloring is a measurable map $f\colon X\to[k]$.  The objects $\tau_r(G,f)$ and $Q_{k,r}(G)$ are defined exactly as in the finite case, with the root sampled according to $\nu$.  We shall use the closed sets
\[
 \calQ_{k,r}(G)=\overline{Q_{k,r}(G)}\subseteq\calP(\calG_{d,k,r}).
\]

\begin{definition}[Local-global convergence for graphings]
A sequence $(G_n)_{n\geq 1}$ of graphings of maximum degree $d$ is \emph{local-global convergent} if, for every $k\geq 1,r\geq 0$, the closed sets $\calQ_{k,r}(G_n)$ converge in Hausdorff distance.
\end{definition}

\begin{definition}[Local-global comparison]
For graphings $G$ and $G'$, write $G\preceq G'$ if
\[
 \calQ_{k,r}(G)\subseteq \calQ_{k,r}(G')
\]
for every $k\geq 1,r\geq 0$.  We say that $G$ and $G'$ are \emph{local-global equivalent} if equality holds for every $k,r$.
\end{definition}

The following compactness theorem of Hatami, Lov\'asz and Szegedy \cite{HLSz} will be used as a black box.

\begin{theorem}[Local-global compactness]\label{thm:HLS}
For fixed maximum degree $d$, the space of local-global equivalence classes of graphings of degree at most $d$ is compact and metrizable.  Equivalently, every sequence of degree-$d$ graphings has a local-global convergent subsequence, and every local-global convergent sequence has a graphing limit $G$ such that, for every $k\geq 1,r\geq 0$,
\[
 \calQ_{k,r}(G_n) \longrightarrow \calQ_{k,r}(G)
\]
in Hausdorff distance.
\end{theorem}

In particular, in this compact metrizable space, if all subsequential limits of a sequence are represented by local-global equivalent graphings, then the sequence itself is local-global convergent.

We now define star observations. Let $\mathsf S_{d,k}$ be the finite set (of isomorphism classes) of rooted $k$-colored stars of maximum degree at most $d$. Equivalently, an element of $\mathsf S_{d,k}$ is a tuple
\[
 (i,n_1,\ldots,n_k),\qquad i\in[k],\quad n_j\in\N,\quad \sum_{j=1}^k n_j\leq d.
\]
For a graphing $G=(X,\nu,E)$ and a measurable coloring $f\colon X\to[k]$, let
\[
 \operatorname{st}_{G,f}(x)\in\mathsf S_{d,k}
\]
be the star with root color $f(x)$ and with $n_i=|\{y\in N_G(x)\,\bigm|\,f(y)=i\}|$.  Let
\[
 \sigma(G,f)
\]
be the distribution of $\operatorname{st}_{G,f}(x)$ for $x\sim\nu$, and define the closed star-statistic set
\[
 \calS_k(G)=\overline{\{\sigma(G,f)\,\bigm|\,f\colon X\to[k]\text{ measurable}\}}
 \subseteq \calP(\mathsf S_{d,k}).
\]
For finite graphs the same definition is used with the uniform measure on vertices.  The set $\calS_k(G)$ is the closed set of all colored degree distributions with $k$ colors.

\begin{lemma}[Continuity of star statistics]\label{lem:star-continuity}
If $G_n$ converges locally-globally to a graphing $G$, then, for every $k\geq 1$,
\[
 \calS_k(G_n)\longrightarrow \calS_k(G)
\]
in Hausdorff distance.
\end{lemma}

\begin{proof}
Let
\[
 \pi_k\colon \calP(\calG_{d,k,1})\to \calP(\mathsf S_{d,k})
\]
be the affine map induced by forgetting all information in a rooted colored radius-one neighbourhood except the color of the root and the multiset of colors of its neighbours.  Then
\[
 \pi_k(Q_{k,1}(G))=\{\sigma(G,f)\,\bigm|\,f\colon X\to[k]\text{ measurable}\}.
\]
Since $\calP(\calG_{d,k,1})$ is compact and $\pi_k$ is continuous,
\[
 \pi_k(\calQ_{k,1}(G))=\overline{\pi_k(Q_{k,1}(G))}=\calS_k(G),
\]
and similarly for $G_n$.  Continuous maps send Hausdorff-convergent compact subsets of a compact metric space to Hausdorff-convergent compact image subsets.  The local-global convergence of $G_n$ to $G$ therefore implies the asserted convergence of $\calS_k(G_n)$ to $\calS_k(G)$.
\end{proof}

We now state the main theorem in its precise form.

\begin{theorem}[Star theorem]\label{thm:main}
Let $(G_n)_{n\geq 1}$ be a sequence of graphings of maximum degree at most $d$.  Then $(G_n)$ is local-global convergent if and only if, for every $k\geq 1$, the closed sets
\[
 \calS_k(G_n)\subseteq\calP(\mathsf S_{d,k})
\]
converge in Hausdorff distance.
\end{theorem}

The forward implication follows from Lemma~\ref{lem:star-continuity}.  The rest of the paper proves the converse.

We shall need one standard measurable coloring fact.

\begin{lemma}[Separated colorings]\label{lem:separated-coloring}
Let $G=(X,\nu,E)$ be a graphing of maximum degree at most $d$, and let $r\geq 1$.  Then there is a measurable coloring
\[
 s\colon X\to[K_{d,r}]
\]
with
\[
 K_{d,r}=d^{2r+1}+1
\]
such that any two distinct vertices lying in a common radius-$r$ ball have distinct colors.
\end{lemma}

\begin{proof}
Form the Borel graph $G^{(2r)}$ on $X$ in which two distinct points are adjacent if their distance in $G$ is at most $2r$.  If two vertices lie in a common radius-$r$ ball, then their distance is at most $2r$.  The maximum degree of $G^{(2r)}$ is at most the crude bound $d^{2r+1}$.  By the bounded-degree Borel coloring theorem of Kechris--Solecki--Todorcevic \cite{KST}, a Borel graph of maximum degree $\Delta$ has a Borel proper coloring with at most $\Delta+1$ colors.  Applying this theorem to $G^{(2r)}$ gives a measurable proper coloring with at most $d^{2r+1}+1$ colors.  Such a coloring separates all vertices lying in a common radius-$r$ ball of $G$.
\end{proof}

\section{Proof of the star theorem}\label{sec:proof}

Throughout this section, all graphings have maximum degree at most the fixed integer $d\geq 2$.  For a graphing $G=(X,\nu,E)$ and $x\in X$, write $N_G(x)$ for the set of neighbours of $x$ and $B_r^G(x)$ for the radius-$r$ ball around $x$.  We omit the superscript when the graphing is clear.

The proof has three steps.  First, we color each vertex by a proposed finite radius-$r$ model and use star statistics to test whether these proposed models agree across edges.  Second, on regions where they agree, walks in the proposed model lift canonically to walks in the graphing.  Finally, equality of closed-walk counts prevents the lift from folding the model, and forces the proposed bounded neighbourhood to be the true one.

Let $W_t^G(x)$ be the set of walks of length $t$ in $G$ starting at $x$, and let $W_t^{0,G}(x)$ be the subset of closed walks.  Define
\[
 c_t(G)=\E_{x\sim\nu}\, |W_t^{0,G}(x)|.
\]
When the graphing is clear, we write $W_t(x)$ and $W_t^0(x)$.

\subsection{Consistent local models}

Let $\calG^*_{d,k,r}\subseteq\calG_{d,k,r}$ denote the set of rooted $k$-colored radius-$r$ graphs in which all vertices have distinct colors. Suppose that $G=(X,\nu,E)$ is a graphing and that
\[
 m\colon X\to\calG^*_{d,k,r}
\]
is a measurable function. We think of $m(x)$ as a rooted colored radius-$r$ model for the ball $B_r^G(x)$ around the vertex $x$. In general, we do not assume that $m(x)$ (without the coloring) is isomorphic to $B_r^G(x)$. Our goal is to identify local properties of $m$, checkable from star observations, that force such an isomorphism.

\begin{definition}[Consistency]\label{def:consistency}
Assume that $r\geq 1$. We say that $m$ is \emph{consistent at $x\in X$} if there is a bijection
\[
 b_x\colon N_{m(x)}(o)\to N_G(x)
\]
from the neighbours of the root in the model $m(x)$ to the neighbours of $x$ in $G$, such that for every $v\in N_{m(x)}(o)$ the rooted colored graph
\[
 \gamma_{r-1,m(x)}(v)
\]
is isomorphic to
\[
 \gamma_{r-1,m(b_x(v))}(o).
\]
Because all vertices of $m(x)$ have distinct colors, the bijection $b_x$ is unique if it exists: the condition forces the root color of $m(b_x(v))$ to be the color of $v$, and the neighbours of the root in $m(x)$ have pairwise distinct colors.
\end{definition}

Thus consistency means that the model assigned to $x$ agrees, one step away from the root, with the models assigned to the actual neighbours of $x$.

\begin{lemma}\label{lem:true-neighbourhoods-consistent}
Let $G=(X,\nu,E)$ be a graphing and let $g\colon X\to[k]$ be a coloring such that every radius-$r$ ball has pairwise distinct colors.  Define
\[
 m(x)=\gamma_{r,G,g}(x)\in\calG^*_{d,k,r}.
\]
Then $m$ is consistent at every $x\in X$.
\end{lemma}

\begin{proof}
The rooted graph $m(x)$ is, by definition, the colored induced radius-$r$ neighbourhood of $x$.  Identify $m(x)$ with the actual ball $B_r(x)$ by the unique color-preserving isomorphism.  Under this identification, neighbours of the root correspond exactly to neighbours of $x$.  If $v$ is such a neighbour and $y$ is the corresponding neighbour of $x$, then the induced radius-$(r-1)$ neighbourhood of $v$ inside $m(x)$ is precisely the induced radius-$(r-1)$ neighbourhood of the root in $m(y)$.  This is exactly the required consistency condition.
\end{proof}

To handle small exceptional sets in graphings, we shall use the following simple estimate, which says that small sets in a bounded-degree graphing have small bounded-radius neighbourhoods. 

\begin{lemma}\label{lem:neighbourhood-measure}
Let $G=(X,\nu,E)$ be a graphing of maximum degree at most $d$, and let $U\subseteq X$ be measurable.  Then
\[
 \nu(B_r(U))\leq (1+d+\cdots+d^r)\nu(U).
\]
\end{lemma}

Note that in this lemma $B_r(U):=\cup_{x\in U}B_r(x)$.

\begin{proof}
It is enough to count how many vertices can be reached from $U$ by walks of length at most $r$.  Let $A_G$ be the adjacency operator of the graphing,
\[
 (A_Gf)(x)=\sum_{y\in N_G(x)} f(y),
\]
and set
\[
 M=I+A_G+\cdots+A_G^r.
\]
The operator $A_G$ is bounded on $L^2(X,\nu)$ with norm at most $d$.  The mass-transport identity gives
\[
 \langle A_G1_U,1_V\rangle=\langle 1_U,A_G1_V\rangle
\]
for measurable sets $U,V$, and hence, by linearity and approximation, $A_G$ is self-adjoint on $L^2(X,\nu)$.  Thus $M$ is self-adjoint.  Since $1_{B_r(U)}\leq M1_U$ pointwise,
\[
 \nu(B_r(U))\leq \langle M1_U,1_X\rangle
 =\langle 1_U,M1_X\rangle.
\]
For every $x$, the value $(M1_X)(x)$ is the number of walks of length at most $r$ starting at $x$, which is at most
\[
 1+d+\cdots+d^r.
\]
The claim follows.
\end{proof}

\begin{lemma}\label{lem:consistency-from-stars}
For every $d,k,r$ and every $\varepsilon>0$ there exists $\delta>0$ with the following property.  Let $G=(X,\nu,E)$ and $G'=(X',\nu',E')$ be graphings, and let
\[
 m\colon X\to\calG^*_{d,k,r},\qquad m'\colon X'\to\calG^*_{d,k,r}
\]
be measurable maps.  Assume that $m'$ is consistent at every point of $X'$ and that, after identifying the finite set $\calG^*_{d,k,r}$ with a color set,
\[
 d_{\TV}\bigl(\sigma(G,m),\sigma(G',m')\bigr)\leq\delta.
\]
Then, for a random $x\in X$, with probability at least $1-\varepsilon$, the map $m$ is consistent at every vertex of $B_r(x)$.
\end{lemma}

\begin{proof}
Let $\mathcal M=\calG^*_{d,k,r}$, regarded as a finite color set.  Consistency at a vertex is determined by the rooted $\mathcal M$-colored star whose root color is $m(x)$ and whose leaf colors are the values of $m$ on the neighbours of $x$: from this finite data one can check whether there is a bijection between the leaves prescribed by the root model and the actual neighbouring model-colors satisfying Definition~\ref{def:consistency}.  Let
\[
 \mathcal B\subseteq \mathsf S_{d,\mathcal M}
\]
be the finite set of star types for which this check fails.  Since $m'$ is everywhere consistent, $\sigma(G',m')(\mathcal B)=0$.  Therefore
\[
 \nu(U)=\sigma(G,m)(\mathcal B)\leq d_{\TV}(\sigma(G,m),\sigma(G',m'))\leq\delta,
\]
where
\[
 U=\{x\in X\,\bigm|\, m\text{ is not consistent at }x\}.
\]
Choose $\delta\leq \varepsilon/(1+d+\cdots+d^r)$.  By Lemma~\ref{lem:neighbourhood-measure},
\[
 \nu(B_r(U))\leq (1+d+\cdots+d^r)\nu(U)\leq\varepsilon.
\]
If $x\notin B_r(U)$, then no vertex of $B_r(x)$ belongs to $U$, so $m$ is consistent at every vertex of $B_r(x)$.  This proves the claim.
\end{proof}

\subsection{Lifting walks from models to the graphing}

The next step shows that, on a ball where the assigned models are consistent, walks in the model can be lifted uniquely to walks in the graphing.

\begin{lemma}\label{lem:walk-lifting}
Let $G=(X,\nu,E)$ be a graphing and let $m\colon X\to\calG^*_{d,k,r}$ be measurable.  Suppose that $m$ is consistent at every vertex of $B_r(x)$.  Then, for every $0\leq t\leq r$, there is a unique bijection
\[
 h_t\colon W_t^{m(x)}(o)\to W_t^G(x)
\]
with the following property.  If
\[
 (v_0,v_1,\ldots,v_t)\in W_t^{m(x)}(o)
\]
and
\[
 h_t(v_0,\ldots,v_t)=(x_0,x_1,\ldots,x_t),
\]
then $x_0=x$, and for every $0\leq i\leq t$,
\[
 \gamma_{r-i,m(x)}(v_i)\cong \gamma_{r-i,m(x_i)}(o)
\]
as rooted colored graphs.  The maps $h_t$ are coherent under restriction to initial subwalks.
\end{lemma}

\begin{proof}
We argue by induction on $t$.  For $t=0$, both walk sets contain only the trivial walk and the assertion is clear.

Assume the statement known for walks of length $t-1$.  Let
\[
 (v_0,\ldots,v_{t-1},v_t)
\]
be a walk in $m(x)$, and write
\[
 h_{t-1}(v_0,\ldots,v_{t-1})=(x_0,\ldots,x_{t-1}).
\]
By the induction hypothesis,
\[
 \gamma_{r-t+1,m(x)}(v_{t-1})\cong \gamma_{r-t+1,m(x_{t-1})}(o).
\]
Since all vertices in the relevant model have distinct colors, this isomorphism is unique.  Let $\kappa$ be this unique isomorphism.  The last vertex $v_t$ is a neighbour of $v_{t-1}$, so $\kappa(v_t)$ is a neighbour of the root in $m(x_{t-1})$.  By consistency at $x_{t-1}$, the canonical bijection
\[
 b_{x_{t-1}}\colon N_{m(x_{t-1})}(o)\to N_G(x_{t-1})
\]
is defined.  Set
\[
 x_t=b_{x_{t-1}}(\kappa(v_t)).
\]
Then $(x_0,\ldots,x_t)$ is a walk in $G$, and the required rooted-neighbourhood compatibility follows directly from the definition of consistency.

At each extension step, neighbours of the current model vertex are matched bijectively with neighbours of the current graphing vertex.  Hence the construction is bijective on walks of length $t$.  Uniqueness and coherence follow from the uniqueness of the color-preserving isomorphisms and of the consistency bijections.
\end{proof}

\begin{lemma}\label{lem:closed-walk-inequality}
Under the assumptions of Lemma~\ref{lem:walk-lifting}, for every $t\leq r$,
\[
 h_t^{-1}\bigl(W_t^{0,G}(x)\bigr)\subseteq W_t^{0,m(x)}(o).
\]
Consequently,
\[
 |W_t^{0,G}(x)|\leq |W_t^{0,m(x)}(o)|.
\]
\end{lemma}

\begin{proof}
Let $(o=v_0,\ldots,v_t)$ be a walk in $m(x)$ whose image under $h_t$ is a closed walk
\[
 (x=x_0,x_1,\ldots,x_t=x)
\]
in $G$.  By Lemma~\ref{lem:walk-lifting},
\[
 \gamma_{r-t,m(x)}(v_t)\cong \gamma_{r-t,m(x_t)}(o)=\gamma_{r-t,m(x)}(o).
\]
The graph $m(x)$ has pairwise distinct colors.  Therefore two vertices of $m(x)$ cannot have isomorphic rooted colored neighbourhoods containing their root colors unless they are the same vertex.  Hence $v_t=o$, and the original walk in $m(x)$ is closed.
\end{proof}

\begin{lemma}[Reverse segment comparison]\label{lem:reverse-segment}
Assume the hypotheses of Lemma~\ref{lem:walk-lifting}.  Let
\[
 p=(o=v_0,\ldots,v_a=w),\qquad
 p'=(o=v'_0,\ldots,v'_b=w)
\]
be two walks in $m(x)$ with $a+b\leq r$.  Suppose that the lift of the closed model walk
\[
 p\overline{p'}=(v_0,\ldots,v_a=v'_b,v'_{b-1},\ldots,v'_0)
\]
is closed in $G$.  Then the endpoints of $h_a(p)$ and $h_b(p')$ coincide.
\end{lemma}

\begin{proof}
Write the lift of $p\overline{p'}$ as
\[
 (z_0,z_1,\ldots,z_{a+b}),
\]
so $z_0=z_{a+b}=x$.  By coherence, the first $a$ steps of this lift are $h_a(p)$, and hence $z_a$ is the endpoint of $h_a(p)$.

Consider the reversed final segment
\[
 (z_{a+b},z_{a+b-1},\ldots,z_a),
\]
which is a graphing walk of length $b$ starting at $x$.  Since $h_b$ is a bijection from model walks of length $b$ starting at $o$ to graphing walks of length $b$ starting at $x$, let
\[
 q=(o=q_0,q_1,\ldots,q_b)
\]
be its preimage under $h_b$.  For each $0\leq j\leq b$, the defining property of $h_b$ gives that the color of $q_j$ is the root color of the model assigned to $z_{a+b-j}$.  On the other hand, applying Lemma~\ref{lem:walk-lifting} to the lift of $p\overline{p'}$ shows that the same root color is the color of $v'_j$.  Since all vertices of $m(x)$ have distinct colors, $q_j=v'_j$ for every $j$.  Thus $q=p'$.

Therefore $h_b(p')$ is precisely the reversed final segment above, and its endpoint is $z_a$, the endpoint of $h_a(p)$.
\end{proof}

The following lemma is the technical heart of the argument.  It says that consistency plus equality of closed-walk counts forces the proposed model to agree with the actual neighbourhood.

\begin{lemma}\label{lem:isomorphism-from-closed-walks}
Let $G=(X,\nu,E)$ be a graphing, let $m\colon X\to\calG^*_{d,k,r}$ be measurable, and let $s\colon X\to[k]$ be the color of the root in $m(x)$.  Suppose that $m$ is consistent at every vertex of $B_r(x)$ and that
\[
 |W_{q}^{0,G}(x)|=|W_{q}^{0,m(x)}(o)|
 \qquad\text{for every }1\leq q\leq r.
\]
If $1\leq t<r/3$, then
\[
 \gamma_{t-1,m(x)}(o)\cong \gamma_{t-1,G,s}(x)
\]
as rooted colored graphs.
\end{lemma}

\begin{proof}
By Lemma~\ref{lem:closed-walk-inequality}, and by equality of the two closed-walk counts, $h_q$ maps closed model walks of length $q$ bijectively onto closed graphing walks of length $q$, for every $q\leq r$.

We first prove that the endpoint in $G$ of a lifted model walk of length at most $t$ depends only on its endpoint in the model.  Let $w$ be a vertex of $m(x)$ at distance at most $t$ from the root, and let
\[
 p=(o=v_0,v_1,\ldots,v_a=w),\qquad
 p'=(o=v'_0,v'_1,\ldots,v'_b=w)
\]
be two model walks with $a,b\leq t$.  The closed model walk $p\overline{p'}$ has length $a+b\leq 2t<r$.  Since closed model walks of this length lift to closed graphing walks, Lemma~\ref{lem:reverse-segment} implies that the endpoints of $h_a(p)$ and $h_b(p')$ coincide.

Thus we may define
\[
 \phi\colon V(\gamma_{t,m(x)}(o))\to V(\gamma_{t,G}(x))
\]
as follows.  For a vertex $v$ in the model, choose any walk from $o$ to $v$ of length at most $t$ and let $\phi(v)$ be the endpoint of its lifted walk in $G$.  The preceding paragraph shows that $\phi$ is well-defined.

The map $\phi$ preserves roots and colors.  Indeed, if the lift of a walk to $v$ ends at $\phi(v)$, Lemma~\ref{lem:walk-lifting} identifies a rooted neighbourhood of $v$ in $m(x)$ with a rooted neighbourhood of the root in $m(\phi(v))$.  In particular the color of $v$ is the root color assigned to $\phi(v)$, namely $s(\phi(v))$.  Since all colors in $m(x)$ are distinct, $\phi$ is injective.

The same closed-walk argument also gives surjectivity onto the actual radius-$t$ ball.  If $y\in B_t^G(x)$ and $\ell\leq t$ is the length of a graphing walk from $x$ to $y$, then the bijection $h_\ell$ has an inverse on walks of length $\ell$; applying this inverse to the chosen walk gives a model walk ending at some vertex $v$ with $\phi(v)=y$.

It remains to check that $\phi$ is a graph isomorphism between the radius-$(t-1)$ neighbourhoods.  If $v,w\in\gamma_{t-1,m(x)}(o)$ are adjacent in the model, then there is a walk from $o$ to $v$ followed by the edge $vw$ of total length at most $t$.  Lifting this walk shows that $\phi(v)$ and $\phi(w)$ are adjacent in $G$.

Conversely, suppose that $y,z\in\gamma_{t-1,G,s}(x)$ are adjacent in $G$.  Choose a walk from $x$ to $y$ of length at most $t-1$, and append the edge $yz$.  This gives a graphing walk of length at most $t$.  Applying the inverse of the corresponding $h_\ell$, where $\ell$ is this length, gives a model walk whose last edge maps to $yz$.  By the definition of $\phi$, the endpoints of this last edge are the unique preimages of $y$ and $z$.  Since both $y$ and $z$ lie within radius $t-1$ of $x$, these preimages lie within radius $t-1$ of $o$.  Hence the edge $yz$ comes from an edge in $\gamma_{t-1,m(x)}(o)$.

Thus $\phi$ is a color-preserving rooted graph isomorphism between the required neighbourhoods.
\end{proof}

\subsection{Closed-walk counts from star data}

\begin{lemma}\label{lem:closed-walk-stability}
For every $d,k,r$ and every $\varepsilon>0$ there exists $\delta>0$ with the following property.  Let $G'=(X',\nu',E')$ be a graphing with a coloring $g'\colon X'\to[k]$ such that every radius-$r$ ball has pairwise distinct colors, and put
\[
 m'(x)=\gamma_{r,G',g'}(x)\in\calG^*_{d,k,r}.
\]
Let $G=(X,\nu,E)$ be another graphing, and let $m\colon X\to\calG^*_{d,k,r}$ be measurable.  If, after identifying $\calG^*_{d,k,r}$ with a finite color set,
\[
 d_{\TV}\bigl(\sigma(G,m),\sigma(G',m')\bigr)\leq\delta,
\]
then, for every $1\leq t\leq r$,

\[
 \E_\nu\left|\,|W_t^{0,m(x)}(o)|-|W_t^{0,G}(x)|\,\right|
 \leq c_t(G')-c_t(G)+\varepsilon,
\]
and consequently
\[
 c_t(G)\leq c_t(G')+\varepsilon.
\]
\end{lemma}

\begin{proof} Observe that at every vertex there are at most $L=L_{d,r}:=d^{r+1}$ closed-walks of lengths at most $r$ in graphs (and graphings) of maximum degree at most $d\geq 2$. 

By Lemma~\ref{lem:true-neighbourhoods-consistent}, $m'$ is consistent everywhere.  By Lemma~\ref{lem:consistency-from-stars}, after choosing $\delta$ sufficiently small, the set
\[
 F=\{x\in X\,\bigm|\, m\text{ is not consistent on every vertex of }B_r(x)\}
\]
has measure as small as prescribed.  We choose $\delta$ so that $\nu(F)\leq \varepsilon/(4L)$.
Fix $1\leq t\leq r$ and write
\[
 A_t(x)=|W_t^{0,m(x)}(o)|,\qquad B_t(x)=|W_t^{0,G}(x)|.
\]
On $X\setminus F$, Lemma~\ref{lem:closed-walk-inequality} gives $B_t(x)\leq A_t(x)$.  Since $0\leq A_t,B_t\leq L$, we have
\[
 \begin{aligned}
 \E|A_t-B_t|
 &=\int_{X\setminus F}(A_t-B_t)\,d\nu+\int_F |A_t-B_t|\,d\nu\\
 &\leq \int_X(A_t-B_t)\,d\nu-\int_F(A_t-B_t)\,d\nu+L\nu(F)\\
 &\leq \E A_t-c_t(G)+2L\nu(F)\\
 &\leq \E A_t-c_t(G)+\varepsilon/2 .
 \end{aligned}
\]

The function $a\mapsto |W_t^{0,a}(o)|$ is a bounded function on $\calG^*_{d,k,r}$ with bound at most $L$. Therefore total variation closeness of the star distributions (and thus closeness of color distributions of $m$ and $m'$) implies, for sufficiently small $\delta$, that
\[
 \left|\E_\nu A_t-\E_{\nu'} |W_t^{0,m'(x)}(o)|\right|\leq\varepsilon/2
\]
simultaneously for all $1\leq t\leq r$.  But $m'(x)$ is the true radius-$r$ neighbourhood of $x$ in $G'$, so
\[
 \E_{\nu'} |W_t^{0,m'(x)}(o)|=c_t(G').
\]
Combining the displayed inequalities proves both conclusions.
\end{proof}

\begin{lemma}\label{lem:closed-walks-equal}
Let $G$ and $G'$ be graphings of maximum degree at most $d$.  Suppose that
\[
 \calS_k(G)=\calS_k(G')
\]
for every $k\geq 1$.  Then
\[
 c_t(G)=c_t(G')
\]
for every $t\geq 1$.
\end{lemma}

\begin{proof}
By symmetry it suffices to prove $c_t(G)\leq c_t(G')$.

Fix $t\geq1$.  By Lemma~\ref{lem:separated-coloring}, choose a coloring
\[
 s\colon X'\to[K_{d,t}]
\]
that separates every radius-$t$ ball in $G'$.  Define
\[
 m'(x)=\gamma_{t,G',s}(x)\in\calG^*_{d,K_{d,t},t}.
\]
Let
\[
 \mathcal M=\calG^*_{d,K_{d,t},t},\qquad a=|\mathcal M|.
\]
Choose a bijection between $\mathcal M$ and $[a]$ and regard $m'$ as an $a$-coloring of $G'$.

The equality $\calS_a(G)=\calS_a(G')$ means that, for every $\delta>0$, there exists a measurable $a$-coloring of $G$ whose star distribution is within $\delta$ of $\sigma(G',m')$.  Decoding this coloring through the chosen bijection gives a measurable map
\[
 m\colon X\to\mathcal M
\]
such that
\[
 d_{\TV}\bigl(\sigma(G,m),\sigma(G',m')\bigr)\leq\delta.
\]
Applying Lemma~\ref{lem:closed-walk-stability} with $r=t$ gives, for every $\varepsilon>0$,
\[
 c_t(G)\leq c_t(G')+\varepsilon.
\]
Letting $\varepsilon\downarrow 0$ yields $c_t(G)\leq c_t(G')$.  The reverse inequality follows by interchanging $G$ and $G'$.
\end{proof}

\subsection{Completion of the proof}

\begin{proof}[Proof of Theorem~\ref{thm:main}]
As noted before, local-global convergence implies convergence of the sets of colored star distributions $\calS_k$ by Lemma~\ref{lem:star-continuity}.

Conversely, assume that $\calS_k(G_n)$ converges for every $k\geq 1$.  We prove that $(G_n)$ is local-global convergent.  By the compactness theorem recalled in Theorem~\ref{thm:HLS}, it is enough to show that any two local-global subsequential limits are local-global equivalent.

Let $G$ and $G'$ be two such subsequential limits.  Since the star data converge along the whole sequence and, by Lemma~\ref{lem:star-continuity}, converge along the two subsequences to the star sets of their local-global limits, we have
\[
 \calS_k(G)=\calS_k(G')
 \qquad\text{for every }k\geq 1.
\]
By Lemma~\ref{lem:closed-walks-equal},
\[
 c_t(G)=c_t(G')
 \qquad\text{for every }t\geq 1.
\]

We prove $\calQ_{k,R}(G')\subseteq\calQ_{k,R}(G)$ for arbitrary $k,R$; the reverse containment follows symmetrically.

Fix $k\geq 1$, $R\geq 1$, and a measurable coloring
\[
 g'\colon X'\to[k]
\]
of $G'$.  Choose an integer $r$ so large that $R+1<r/3$.  Let
\[
 s_r\colon X'\to[K_{d,r}]
\]
be a coloring that separates all radius-$r$ balls in $G'$, and form the product coloring
\[
 \widetilde g'=(g',s_r)\colon X'\to[k]\times[K_{d,r}].
\]
Let
\[
 m'(x)=\gamma_{r,G',\widetilde g'}(x)
 \in\calG^*_{d,kK_{d,r},r}.
\]
By Lemma~\ref{lem:true-neighbourhoods-consistent}, $m'$ is consistent everywhere.

Put
\[
 \mathcal M=\calG^*_{d,kK_{d,r},r},\qquad a=|\mathcal M|,
\]
choose a bijection $\mathcal M\leftrightarrow[a]$, and regard $m'$ as an $a$-coloring of $G'$.  Since $\calS_a(G)=\calS_a(G')$, for every $\delta>0$ there is a measurable map
\[
 m\colon X\to\mathcal M
\]
such that
\[
 d_{\TV}\bigl(\sigma(G,m),\sigma(G',m')\bigr)\leq\delta.
\]
Let $\widetilde g\colon X\to[k]\times[K_{d,r}]$ be the root color of $m(x)$, and let $g$ be its first coordinate.

By Lemma~\ref{lem:consistency-from-stars}, for all but $o_\delta(1)$ of the points $x\in X$, the map $m$ is consistent on all of $B_r(x)$.  By Lemmas~\ref{lem:closed-walk-stability} and~\ref{lem:closed-walks-equal}, for each $1\leq q\leq r$,
\[
 \E_\nu\left|\,|W_q^{0,m(x)}(o)|-|W_q^{0,G}(x)|\,\right|=o_\delta(1).
\]
Let
\[
 D_q=\{x\,\bigm|\, |W_q^{0,m(x)}(o)|\neq |W_q^{0,G}(x)|\}.
\]
The discrepancy is integer-valued, hence $\nu(D_q)\leq \E  \bigl|\,|W_q^{0,m(x)}(o)|-|W_q^{0,G}(x)|\,\bigr|=o_\delta(1)$.  Taking the union over the finitely many $q=1,\ldots,r$, we obtain a set of measure $o_\delta(1)$ outside which
\[
 |W_q^{0,m(x)}(o)|=|W_q^{0,G}(x)|
 \qquad\text{for all }1\leq q\leq r,
\]
and $m$ is consistent on $B_r(x)$.  On this good set, Lemma~\ref{lem:isomorphism-from-closed-walks}, applied with $t=R+1$, gives
\[
 \gamma_{R,m(x)}(o)
 \cong
 \gamma_{R,G,\widetilde g}(x).
\]

The root marginal of $\sigma(G,m)$ is the law of $m(x)$, and the root marginal of $\sigma(G',m')$ is the law of $m'(x')$.  Therefore these two laws are $o_\delta(1)$-close.  Applying the deterministic truncation map
\[
 \mathcal M\to \calG_{d,kK_{d,r},R},\qquad M\mapsto \gamma_{R,M}(o),
\]
we see that the distribution of $\gamma_{R,m(x)}(o)$ is $o_\delta(1)$-close to the distribution of
\[
 \gamma_{R,G',\widetilde g'}(x').
\]
Thus
\[
 \tau_R(G,\widetilde g)
\]
can be made arbitrarily close to
\[
 \tau_R(G',\widetilde g').
\]
Finally let
\[
 \Pi\colon \calP(\calG_{d,kK_{d,r},R})\to \calP(\calG_{d,k,R})
\]
be the affine continuous map induced by forgetting the auxiliary separating color.  Applying $\Pi$, we conclude that
\[
 \tau_R(G,g)
\]
can be made arbitrarily close to
\[
 \tau_R(G',g').
\]
Thus $\tau_R(G',g')\in\calQ_{k,R}(G)$.  Since $g'$ was arbitrary, $\calQ_{k,R}(G')\subseteq\calQ_{k,R}(G)$.

Interchanging the roles of $G$ and $G'$ gives the reverse containment.  Hence the two graphings are local-global equivalent.  Therefore all local-global subsequential limits of $(G_n)$ are equivalent, and $(G_n)$ is local-global convergent.
\end{proof}

\section{The colored cherry metric}\label{sec:cherry}

In this section we introduce a second observation which, at first sight, appears weaker than the colored star observation. We shall show that it nevertheless generates the same topology. Combined with Theorem~\ref{thm:main}, this implies that colored cherry statistics are also equivalent to local-global convergence.

The point of the argument is the following. A colored star records the whole multiset of colors around a vertex. A colored cherry records only two randomly chosen neighbours of the root, with special deterministic conventions in degrees zero and one. Thus a single cherry observation does not see the whole degree vector. However, if arbitrary auxiliary colorings are allowed, we may color each vertex by a proposed local star. The cherry observation then tests whether pairs of actual neighbours are compatible with this proposed star. This pairwise information gives an injective matching from the actual neighbours into the leaves prescribed by the model. A symmetric argument forces equality of average degrees, and hence this injection must be a bijection on most vertices. At that point the proposed star is a genuine colored star.

\subsection{Cherry observations}

Let $\mathsf{Ch}_k$ denote the finite set of isomorphism classes of the following rooted $k$-colored graphs:
\begin{enumerate}[label=(\roman*)]
 \item a single rooted colored vertex;
 \item a rooted colored edge, with the root at one endpoint;
 \item a colored path of length two rooted at the central vertex, equivalently a rooted two-star with two unordered leaves.
\end{enumerate}
In the third case the two leaves are not ordered.

Let $G=(X,\nu,E)$ be a graphing of maximum degree at most $d$, and let $f\colon X\to[k]$ be a measurable coloring. We define a probability distribution
\[
 \chi(G,f)\in\calP(\mathsf{Ch}_k)
\]
as follows. First choose a random vertex $x\in X$ according to $\nu$.
\begin{itemize}
 \item If $\deg_G(x)=0$, the outcome is the one-point rooted colored graph consisting of $x$.
 \item If $\deg_G(x)=1$, with unique neighbour $y$, the outcome is the rooted colored edge $xy$, rooted at $x$.
 \item If $\deg_G(x)\geq 2$, choose two distinct neighbours $y,z\in N_G(x)$ uniformly among all unordered pairs, and output the rooted colored two-star with centre $x$ and leaves $y,z$.
\end{itemize}
For finite graphs the same definition is used with the uniform measure on vertices.  More generally, if the coloring takes values in any finite set, we identify that set with $[a]$ and use the same notation $\chi(G,f)$.

Define the raw cherry set
\[
 C_k(G)=\{\chi(G,f)\,\bigm|\, f\colon X\to[k]\text{ measurable}\}\subseteq\calP(\mathsf{Ch}_k)
\]
and let $\overline C_k(G)$ be its closure in total variation distance.

\begin{definition}[colored cherry convergence]
A sequence of graphings $(G_n)_{n\geq 1}$ of maximum degree at most $d$ is \emph{colored-cherry convergent} if, for every $k\geq 1$, the closed sets $\overline C_k(G_n)$ converge in Hausdorff distance.
\end{definition}

The goal of this section is to prove the following theorem.

\begin{theorem}[Cherry, star and local-global equivalence]\label{thm:cherry-equivalence}
Let $(G_n)_{n\geq 1}$ be a sequence of graphings of maximum degree at most $d\geq 2$. Then the following are equivalent:
\begin{enumerate}[label=(\alph*)]
 \item $(G_n)$ is local-global convergent;
 \item for every $k\geq 1$, the sets $\calS_k(G_n)$ of colored star statistics converge;
 \item for every $k\geq 1$, the sets $\overline C_k(G_n)$ of colored cherry statistics converge.
\end{enumerate}
\end{theorem}

The equivalence of (a) and (b) is Theorem~\ref{thm:main}. It remains to compare stars and cherries.

\subsection{Stars determine cherries}

The easy direction is that a colored star determines the distribution of the cherry sampled from its root.

Let $A\in\mathsf S_{d,k}$ be a rooted colored star. If the root has degree $0$, define $\Theta_k(A)$ to be the corresponding one-point cherry. If the root has degree $1$, define $\Theta_k(A)$ to be the corresponding rooted colored edge. If the root has degree at least $2$, define $\Theta_k(A)$ to be the probability distribution on rooted colored two-stars obtained by choosing two distinct leaves of $A$ uniformly at random. Extending linearly gives an affine continuous map
\[
 \Theta_k\colon \calP(\mathsf S_{d,k})\to\calP(\mathsf{Ch}_k).
\]

\begin{lemma}\label{lem:star-to-cherry}
For every graphing $G$ and every measurable coloring $f\colon X\to[k]$,
\[
 \chi(G,f)=\Theta_k\bigl(\sigma(G,f)\bigr).
\]
Consequently,
\[
 \Theta_k(\calS_k(G))=\overline C_k(G),
\]
and convergence of $\calS_k(G_n)$ for every $k$ implies convergence of $\overline C_k(G_n)$ for every $k$.
\end{lemma}

\begin{proof}
Condition on the rooted colored star $\operatorname{st}_{G,f}(x)$. The cherry sampling rule depends only on this star: for degree $0$ and degree $1$ there is no further choice, while for degree at least $2$ the two leaves are chosen uniformly among the unordered pairs of leaves of the star. This is exactly the definition of $\Theta_k$.

Let
\[
 R_k(G)=\{\sigma(G,f)\,\bigm|\,f\colon X\to[k]\text{ measurable}\}
\]
be the raw star set.  The identity above gives $\Theta_k(R_k(G))=C_k(G)$.  Since $\calP(\mathsf S_{d,k})$ is compact, $\calS_k(G)$ is compact, and the continuous image $\Theta_k(\calS_k(G))$ is compact and hence closed.  Therefore
\[
 \Theta_k(\calS_k(G))=\overline{\Theta_k(R_k(G))}=\overline C_k(G).
\]
Finally, continuous maps send Hausdorff-convergent compact sets to Hausdorff-convergent compact image sets.
\end{proof}

\subsection{Model colors and cherry-consistency}

We now prove the converse direction. Throughout this subsection, fix a finite color set $[k]$.  Whenever a map takes values in a finite model set, such as $\mathsf S_{d,k}$, we identify that finite set with a color set when forming cherry statistics.

A rooted colored star $A\in\mathsf S_{d,k}$ has a root color, denoted $\rho(A)$, and a multiset of leaf colors. For $i\in[k]$, let
\[
 n_i(A)
\]
be the number of leaves of color $i$ in $A$, and let
\[
 \deg(A)=\sum_{i=1}^k n_i(A)
\]
be the degree of its root.

Now let $G=(X,\nu,E)$ be a graphing and let
\[
 m\colon X\to\mathsf S_{d,k}
\]
be a measurable assignment of a proposed colored star to each vertex. We write
\[
 \rho_m(x)=\rho(m(x))
\]
for the proposed root color at $x$. For $i\in[k]$, define the actual neighbour count
\[
 a_i^m(x)=|\{y\in N_G(x)\,\bigm|\,\rho_m(y)=i\}|.
\]
Thus $m$ is a genuine star assignment precisely when
\[
 a_i^m(x)=n_i(m(x))\qquad\text{for all }i
\]
for almost every $x$.

\begin{definition}[Cherry-consistency]\label{def:cherry-consistency}
The assignment $m$ is called \emph{cherry-consistent at $x$} if the following three conditions hold.
\begin{enumerate}[label=(\roman*)]
 \item If $\deg_G(x)=0$, then $\deg(m(x))=0$.
 \item If $\deg_G(x)=1$, with unique neighbour $y$, then $m(x)$ has at least one leaf of color $\rho_m(y)$.
 \item If $\deg_G(x)\geq 2$, then for every two distinct neighbours $y,z\in N_G(x)$, the multiset
 \[
  \{\rho_m(y),\rho_m(z)\}
 \]
 is contained in the multiset of leaf colors of $m(x)$.
\end{enumerate}
\end{definition}

This definition says exactly that every cherry which can be sampled around $x$ is compatible with the proposed star $m(x)$.

\begin{lemma}\label{lem:cherry-test-consistency}
For every $d,k$ and every $\varepsilon>0$ there exists $\delta>0$ with the following property. Let $G=(X,\nu,E)$ and $G'=(X',\nu',E')$ be graphings, and let
\[
 m\colon X\to\mathsf S_{d,k},\qquad m'\colon X'\to\mathsf S_{d,k}
\]
be measurable assignments. Assume that $m'$ is the true star assignment of its root-color map, that is,
\[
 m'(x')=\operatorname{st}_{G',\rho_{m'}}(x')
 \qquad\text{for almost every }x'\in X',
\]
and assume that
\[
 d_{\TV}\bigl(\chi(G,m),\chi(G',m')\bigr)\leq\delta,
\]
where $m$ and $m'$ are viewed as finite-valued colorings. Then $m$ is cherry-consistent outside a set of measure at most $\varepsilon$.
\end{lemma}

\begin{proof}
Let $\mathcal M=\mathsf S_{d,k}$, viewed as the color set used in the cherry statistic.  Define a subset
\[
 \mathcal A\subseteq \mathsf{Ch}_{\mathcal M}
\]
of admissible cherry outcomes as follows.  A one-point outcome with root color $A\in\mathcal M$ is admissible if $\deg(A)=0$.  A rooted-edge outcome with root color $A$ and neighbour color $B$ is admissible if $A$ has at least one leaf of color $\rho(B)$.  A two-star outcome with root color $A$ and leaf colors $B,C$ is admissible if the multiset $\{\rho(B),\rho(C)\}$ is contained in the multiset of leaf colors of $A$.  This is a finite condition depending only on the cherry outcome.

Since $m'$ is a true star assignment, every cherry outcome produced by $(G',m')$ belongs to $\mathcal A$.  Therefore
\[
 \chi(G,m)(\mathsf{Ch}_{\mathcal M}\setminus\mathcal A)\leq
 d_{\TV}\bigl(\chi(G,m),\chi(G',m')\bigr)\leq\delta.
\]
%Let
%\[
% B_d=\max\left(1,\binom d2\right).
%\]
If $m$ is not cherry-consistent at a vertex $x$, then at least one inadmissible cherry can be sampled from $x$.  The conditional probability of sampling such a bad cherry is at least $\binom d2^{-1}$: it is $1$ in degrees $0$ and $1$, and at least $\binom d2^{-1}$ in degree at least $2$.  Hence the measure of the set of vertices where $m$ is not cherry-consistent is at most $\binom d2\delta$.  Choosing $\delta\leq \varepsilon/\binom d2$ proves the lemma.
\end{proof}

Cherry-consistency alone does not always give full consistency: a proposed star might contain extra leaves. The next lemma explains why this is the only possible defect once the colors separate neighbours.

\begin{lemma}\label{lem:cherry-consistency-injection}
Assume that $m\colon X\to\mathsf S_{d,k}$ is cherry-consistent at $x$ and that no two distinct neighbours of $x$ have the same proposed root color $\rho_m$. Then there is an injective map
\[
 \iota_x\colon N_G(x)\to N_{m(x)}(o)
\]
such that, for every $y\in N_G(x)$, the leaf $\iota_x(y)$ has color $\rho_m(y)$. In particular,
\[
 \deg_G(x)\leq \deg(m(x)).
\]
If equality holds, then $m(x)$ is the true colored star of $x$ with respect to the coloring $\rho_m$, namely
\[
 m(x)=\operatorname{st}_{G,\rho_m}(x).
\]
\end{lemma}

\begin{proof}
If $\deg_G(x)=0$, the conclusion is immediate from cherry-consistency. If $\deg_G(x)=1$, cherry-consistency says that the color of the unique neighbour occurs among the leaves of $m(x)$, so the required injection exists.

Assume now that $\deg_G(x)\geq 2$. Let $y\in N_G(x)$. Choose another neighbour $z\neq y$. By cherry-consistency, the two-color multiset $\{\rho_m(y),\rho_m(z)\}$ is contained in the leaf-color multiset of $m(x)$. Hence $m(x)$ has at least one leaf of color $\rho_m(y)$. Since the colors $\rho_m(y)$ are pairwise distinct as $y$ ranges over $N_G(x)$, these leaves can be chosen distinctly. This gives the desired injection.

The inequality of degrees follows from the injection. If equality holds, the injection is a bijection. Therefore the multiset of colors of actual neighbours of $x$ is exactly the multiset of leaf colors prescribed by $m(x)$, and the root color is, by definition, $\rho_m(x)$. Hence $m(x)$ is precisely the rooted colored star of $x$ for the coloring $\rho_m$.
\end{proof}

\begin{lemma}\label{lem:degree-inequality-from-cherries}
Let $G=(X,\nu,E)$ and $G'=(X',\nu',E')$ be graphings of maximum degree at most $d$. Suppose that
\[
 \overline C_k(G)=\overline C_k(G')
 \qquad\text{for every }k\geq 1.
\]
Then $G$ and $G'$ have the same average degree:
\[
 \int_X \deg_G(x)\,d\nu(x)=\int_{X'}\deg_{G'}(x')\,d\nu'(x').
\]
\end{lemma}

\begin{proof}
We first prove that the average degree of $G$ is at most that of $G'$. Choose, by Lemma~\ref{lem:separated-coloring}, a measurable coloring
\[
 s'\colon X'\to[K_{d,1}]
\]
which separates all radius-one balls in $G'$. Let
\[
 m'(x')=\operatorname{st}_{G',s'}(x')
\]
be the true separated star at $x'$. Put
\[
 \mathcal M=\mathsf S_{d,K_{d,1}},\qquad a=|\mathcal M|,
\]
choose a bijection $\mathcal M\leftrightarrow[a]$, and regard $m'$ as an $a$-coloring of $G'$.

By the equality of the closed cherry sets, for every $\delta>0$ there is a measurable map
\[
 m\colon X\to\mathcal M
\]
such that
\[
 d_{\TV}\bigl(\chi(G,m),\chi(G',m')\bigr)\leq\delta.
\]
By Lemma~\ref{lem:cherry-test-consistency}, $m$ is cherry-consistent outside a set of measure $O_d(\delta)$.

The cherry outcomes of $(G',m')$ never contain two neighbours of the root with the same proposed root color, because $s'$ separates neighbours of every vertex. Hence, by total variation closeness, the probability that a random cherry of $(G,m)$ contains two sampled neighbours with the same proposed root color is $O(\delta)$. If a vertex has two neighbours with the same proposed root color, then such a pair is sampled with probability at least $\binom d2^{-1}$. Thus, outside another set of measure $O_d(\delta)$, no two neighbours of $x$ have the same proposed root color.

On the good set, Lemma~\ref{lem:cherry-consistency-injection} gives
\[
 \deg_G(x)\leq \deg(m(x)).
\]
The exceptional set contributes $O_d(\delta)$ to the expectation, because degrees are uniformly bounded. Therefore
\[
 \int_X\deg_G(x)\,d\nu(x)\leq \int_X\deg(m(x))\,d\nu(x)+O_d(\delta).
\]
The quantity $\deg(m(x))$ is a function only of the root color $m(x)$, hence its expectation is determined by the root marginal of the cherry distribution. Since the cherry distributions of $(G,m)$ and $(G',m')$ are within $\delta$,
\[
 \int_X\deg(m(x))\,d\nu(x)=\int_{X'}\deg(m'(x'))\,d\nu'(x')+O_d(\delta).
\]
But $m'$ is the true star assignment, so $\deg(m'(x'))=\deg_{G'}(x')$. Letting $\delta\downarrow0$ yields
\[
 \int_X\deg_G(x)\,d\nu(x)\leq \int_{X'}\deg_{G'}(x')\,d\nu'(x').
\]
Interchanging the roles of $G$ and $G'$ gives the reverse inequality.
\end{proof}

\subsection{Cherries determine stars}

We now show that equality of all cherry statistics implies equality of all star statistics.

\begin{lemma}\label{lem:cherries-determine-stars}
Let $G$ and $G'$ be graphings of maximum degree at most $d$. Suppose that
\[
 \overline C_k(G)=\overline C_k(G')
 \qquad\text{for every }k\geq 1.
\]
Then
\[
 \calS_k(G)=\calS_k(G')
 \qquad\text{for every }k\geq 1.
\]
\end{lemma}

\begin{proof}
It is enough, by symmetry, to prove
\[
 \calS_k(G')\subseteq\calS_k(G)
\]
for every $k$.

Fix a measurable coloring $g'\colon X'\to[k]$. We shall approximate $\sigma(G',g')$ by star statistics of colorings of $G$.

Let $s'\colon X'\to[K_{d,1}]$ be a coloring separating all radius-one balls in $G'$, and define the product coloring
\[
 \widetilde g'=(g',s')\colon X'\to[k]\times[K_{d,1}].
\]
Put
\[
 m'(x')=\operatorname{st}_{G',\widetilde g'}(x').
\]
Thus $m'$ is a true separated star assignment. Let
\[
 \mathcal M=\mathsf S_{d,kK_{d,1}},\qquad a=|\mathcal M|,
\]
choose a bijection $\mathcal M\leftrightarrow[a]$, and regard $m'$ as an $a$-coloring of $G'$.

By equality of the cherry sets, for every $\delta>0$ there exists a measurable map
\[
 m\colon X\to\mathcal M
\]
such that
\[
 d_{\TV}\bigl(\chi(G,m),\chi(G',m')\bigr)\leq\delta.
\]
Let
\[
 \widetilde g(x)=\rho(m(x))\in[k]\times[K_{d,1}]
\]
be the proposed root color, and let $g(x)$ be its first coordinate.

By Lemma~\ref{lem:cherry-test-consistency}, $m$ is cherry-consistent outside a set of measure $O_d(\delta)$. As in the proof of Lemma~\ref{lem:degree-inequality-from-cherries}, the separation visible in the cherry distribution of $(G',m')$ implies that, outside another set of measure $O_d(\delta)$, no two neighbours of $x$ have the same proposed root color $\widetilde g$.

Therefore Lemma~\ref{lem:cherry-consistency-injection} gives, outside a set $B_\delta$ of measure $O_d(\delta)$,
\[
 \deg_G(x)\leq\deg(m(x)).
\]
By Lemma~\ref{lem:degree-inequality-from-cherries}, $G$ and $G'$ have the same average degree. Moreover, total variation closeness of the cherry distributions gives
\[
 \int_X\deg(m(x))\,d\nu(x)=\int_{X'}\deg(m'(x'))\,d\nu'(x')+O_d(\delta)=
\]
\[=\int_{X'}\deg_{G'}(x')\,d\nu'(x')+O_d(\delta)
 =\int_X\deg_G(x)\,d\nu(x)+O_d(\delta).
\]
On $X\setminus B_\delta$ the difference $\deg(m(x))-\deg_G(x)$ is non-negative and at most $d$.  The preceding expectation estimate, together with the bound $d\nu(B_\delta)=O_d(\delta)$, gives
\[
 \int_{X\setminus B_\delta}(\deg(m(x))-\deg_G(x))\,d\nu(x)=O_d(\delta).
\]
Since this difference is integer-valued, the set on which it is positive has measure $O_d(\delta)$. Hence, outside a set of measure $O_d(\delta)$,
\[
 \deg(m(x))=\deg_G(x).
\]
Lemma~\ref{lem:cherry-consistency-injection} now implies that, outside a set of measure $O_d(\delta)$,
\[
 m(x)=\operatorname{st}_{G,\widetilde g}(x).
\]

The root marginal of $\chi(G,m)$ is the distribution of the model $m(x)$, and the root marginal of $\chi(G',m')$ is the distribution of $m'(x')$. Hence
\[
 d_{\TV}\bigl(\operatorname{law}(m(x)),\operatorname{law}(m'(x'))\bigr)=O(\delta).
\]
Combining this with the previous paragraph gives
\[
 \sigma(G,\widetilde g)\longrightarrow \sigma(G',\widetilde g')
\]
in total variation as $\delta\downarrow0$. Finally, forgetting the auxiliary separating coordinate $s'$ is a continuous affine projection from $(kK_{d,1})$-colored stars to $k$-colored stars. Therefore
\[
 \sigma(G,g)
\]
can be made arbitrarily close to
\[
 \sigma(G',g').
\]
This proves $\sigma(G',g')\in\calS_k(G)$. Since $g'$ was arbitrary, $\calS_k(G')\subseteq\calS_k(G)$.
\end{proof}

\begin{proof}[Proof of Theorem~\ref{thm:cherry-equivalence}]
The equivalence of local-global convergence and convergence of all colored star sets is Theorem~\ref{thm:main}. By Lemma~\ref{lem:star-to-cherry}, convergence of all colored star sets implies convergence of all colored cherry sets.

Conversely, suppose that $(G_n)$ is colored-cherry convergent. Let $G$ and $G'$ be two subsequential local-global limits of $(G_n)$, whose existence is guaranteed by Theorem~\ref{thm:HLS}. Since local-global convergence along a subsequence implies star convergence by Lemma~\ref{lem:star-continuity}, and hence cherry convergence by Lemma~\ref{lem:star-to-cherry}, the cherry sets along these two subsequences converge respectively to the cherry sets of their limits.  Because the cherry sets converge along the whole sequence, we have
\[
 \overline C_k(G)=\overline C_k(G')
 \qquad\text{for every }k\geq 1.
\]
By Lemma~\ref{lem:cherries-determine-stars},
\[
 \calS_k(G)=\calS_k(G')
 \qquad\text{for every }k\geq 1.
\]
Theorem~\ref{thm:main} then implies that $G$ and $G'$ are local-global equivalent. Hence all local-global subsequential limits of $(G_n)$ are equivalent, so $(G_n)$ is local-global convergent. This proves the theorem.
\end{proof}

\section*{Acknowledgements}
The author acknowledges the use of ChatGPT as a writing aid during the preparation of this manuscript, specifically for improving grammar, style, and exposition. All mathematical ideas, statements, and proofs were developed by the author and were already present in a preliminary version of the paper written around 2020. This research was supported by the NKFIH excellence 154126 grant.

\end{document}